\documentclass[11pt]{amsart}

\advance\hoffset by -0.5 truecm

\usepackage{amssymb}
\usepackage{amscd}
\usepackage{verbatim}
\usepackage{epsfig}
\usepackage{euscript}
\usepackage[colorlinks=true]{hyperref}
\usepackage{cite}
\hypersetup{urlcolor=blue, citecolor=red}
\usepackage{color}
\usepackage{mathtools}
\mathtoolsset{showonlyrefs}
\usepackage{tikz-cd}

\newtheorem{theorem}{Theorem}%[section]
\newtheorem{lemma}{Lemma}
\newtheorem{corollary}{Corollary} 
\newtheorem{proposition} {Proposition}

\theoremstyle{definition}
\newtheorem{definition}{Definition}

\theoremstyle{remark}
\newtheorem{remark}{Remark}

\def \R{\mathbb R}
\def \p{\partial}

\def \<{\langle}
\def \>{\rangle}

\DeclareMathOperator{\tr}{tr}
\renewcommand{\d}{{d}}

\newcommand{\supp}{\operatorname{supp}}
\newcommand{\Ker}{\operatorname{Ker}}

\newcommand{\id}{\operatorname{id}}
\let \div \undefined
\DeclareMathOperator{\div}{div}

\renewcommand{\r}{\eqref}
\newcommand{\bo}{{\partial \Omega}}

\begin{document}

%\title[Minimal surface and Radon transform]{Geometric inverse problem for minimal surfaces and Radon transform of tensor fields}
\title[The linearized minimal surfaces problem]{The linearized minimal surfaces problem}

\author[Haim Grebnev]{Haim Grebnev}
\address{Department of Mathematics, Purdue University, West Lafayette, IN 47907, USA}
\email{hgrebnev@purdue.edu}

\author[Plamen Stefanov]{Plamen Stefanov}
\address{Department of Mathematics, Purdue University, West Lafayette, IN 47907, USA}
%\thanks{P.S. partly supported by  NSF  Grant DMS-2154489}
\email{stefanov@math.purdue.edu}

\thanks{P.S. partly supported by  NSF  Grants DMS-2154489 and DMS-2452757.}

\author[Gunther Uhlmann]{Gunther Uhlmann}
\address{Department of Mathematics, University of Washington, Seattle, WA 98195, USA}
\email{gunther@math.washington.edu}

\thanks{G.U. is partially supported by NSF}

\author[Hanming Zhou]{Hanming Zhou}
\address{Department of Mathematics, University of California Santa Barbara, Santa Barbara, CA 93106-3080, USA}
\email{hzhou@math.ucsb.edu}

\thanks{H.Z. is partly supported by NSF grant DMS-2408369 and Simons Foundation Travel Support for Mathematicians MPS-TSM-00008046.}

%\date{\today}
\date{}

\begin{abstract}
We characterize the kernel of the linearization $R$ of the minimal surface problem about the Euclidean metric in a bounded smooth domain $\Omega\subset\R^n$, $n\ge2$,  with the background minimal surfaces being the Euclidean planes.  We show that, in the whole-space Euclidean decomposition, the kernel consists of potential fields and TT fields. For bounded domains, a similar phenomenon
appears with additional boundary coupling conditions; in particular, the TT part may be coupled to a harmonic conformal component.
\end{abstract}
\maketitle

%--------------------------------------------------------------------------------

\section{Introduction}

The nonlinear minimal surface problem asks whether we can determine a Riemannian metric $g$ on a compact Riemannian manifold $M$ with boundary, up to an isometry fixing the boundary $\partial M$, from the knowledge of the areas of all minimal (hyper)surfaces $\Sigma$ with given ``loops'' $\Sigma\cap \partial M$. For the recent progress with this problem, and its relevance to physics, we refer to   \cite{AlexakisBN,busch2025_minimal}, and the references therein.

The purpose of this note is to analyze the kernel of the linearization $R$ of the minimal surface problem about the Euclidean metric in a bounded smooth domain $\Omega\subset\R^n$, $n\ge 2$,  with the background minimal surfaces chosen to be the Euclidean planes. This linearization was computed in  \cite{busch2025_minimal} to be a certain Radon transform of tensor fields of order two over the minimal surfaces chosen, see Definition~\ref{def1} in the appendix. In the Euclidean case under consideration, it reduces to 
\[
Rf(\Sigma):= \int_{\Sigma} \tr_{\Sigma}(f)\, \d \mu,
\]
where $\Sigma$ runs over all hyperplanes,  $\tr_{\Sigma}(f)$ is the trace of $f$ restricted to $\Sigma$,  and $\d \mu$ is the surface measure. 
Of course, even in $\mathbb R^n$ the shape of minimal hypersurfaces can be quite complicated, such as catenoids, or helicoids, for example. Here we only consider the trivial minimal hypersurfaces, namely the hyperplanes in $\mathbb R^n$. Notice that the (oriented) hyperplanes in $\mathbb R^n$ can be parameterized by $(s,\omega)\in \mathbb R\times \mathbb S^{n-1}$ as $\{x\in\mathbb R^n: x\cdot \omega=s\}$. 
We first study the Radon transform of tensor fields defined on the whole Euclidean space. We can write
\begin{equation}\label{R2}
Rf(s,\omega) = \int_{x\cdot \omega=s} (\tr f - f_{ij} \omega^i\omega^j) \, \d \mu, 
\end{equation}
where $\tr f=f_{ij}\delta^{ij}=\sum_{i=1}^n f_{ii}$ with $\delta^{ij}$ the Kronecker delta function. 
It is straightforward to see that potential fields $f=d^s v$ (here $d^s$ is the symmetric differentiation), for one-form $v$ vanishing on the boundary $\bo$, are in $\Ker R$ even in the general Riemannian case, see the appendix. This is expected since the nonlinear minimal surface problem has the diffeomorphism gauge invariance, and potential fields linearize it.
We show however, that $\Ker R$ contains TT tensor fields as well (divergence-free and trace-free) satisfying certain boundary conditions. This is unexpected since it was proven in \cite{AlexakisBN}, that the nonlinear problem has a unique solution up to the diffeomorphism gauge. Therefore, there is no  second gauge group of transformations preserving the data, which TT tensors would linearize. On the other hand, we study a formally determined problem, while in \cite{AlexakisBN}, the nonlinear problem is overdetermined by infinitely many dimensions. Perhaps we can view the TT subspace as an instability one. 

In a recent paper \cite{busch2025_minimal}, it was shown that $R$ is invertible when $f$ is kept in a fixed conformal class with a background analytic metric; which also leads to uniqueness and stability results for the nonlinear problem. In the conformal case, $R$ reduces to a generalized Radon transform of a function over minimal surfaces, which is elliptic, at least under the Bolker condition, see also \cite{ ZH_surfaces,  AlexakisBN, Beylkin}, where the surfaces do not need to be minimal. 

We present several characterizations of $\Ker R$. First, using  a Fourier transform based analysis in section~\ref{sec_F}, we show that once we decompose $f$, originally supported in $\Omega$,  into 
\[
f=d^s v+f^\mathsf{TT}+\lambda \bf{e},
\]
with the potential part $d^s v$, the TT part $f^\mathsf{TT}$ and the conformal part $\lambda \bf{e}$ \textit{in the whole $\R^n$}, then $\Ker R$ consists of the first two parts, which in general extend their supports to the whole $\R^n$. We turn our attention next to a characterization of $\Ker R$ without ``escaping'' to $\R^n$. In Lemma~\ref{lemma1}, we prove that $\Ker R$ coincides with the kernel of the simple differential operator 
\begin{equation}\label{P}
Pf := -\Delta \tr f+ \p^j \p^i f_{ij}
\end{equation}  
when we view $f$ extended as zero outside $\Omega$. The operator $P$ happens to be the linearization of the scalar curvature about the Euclidean metric corresponding to an infinitesimal perturbation $f$ (a covariant symmetric 2-tensor field) of the metric tensor. The extension as zero could generate delta type terms at $\bo$, and their vanishing gives us boundary conditions. We exploit them in Theorem~\ref{thm_M1} and Theorem~\ref{thm_M2} to formulate characterizations of $\Ker R$ in terms of the standard decompositions of tensor fields of order two in $\Omega$. Even though those theorems give necessary and sufficient conditions, they still leave unanswered the question whether one can simplify the description even further.

\section{Fourier transform based analysis in $\R^n$}\label{sec_F} 
Note first that $Rf$ is well defined on $L^1(\R^n)$ with values in $L^\infty(S^{n-1}_\omega ;\, L^1(\R_s))$. It is also well-defined on $L^2(\R^n)$ so that to $ |D_s|^{(n-1)/2} Rf\in L^2_e(\R\times S^{n-1};\,  ds\,d\omega)$, where the subscript $e$ stands for the subspace of even functions of $(s,\omega)$. That map is actually unitary after scaling by a proper constant.

We start with a  Fourier slice theorem.
%$$
%R_0(f)(s,\omega):=\int_{x\cdot \omega=s} f(x)\, \d\mu(x)
%$$
%with $\d\mu(x)$ the measure on the hyperplane.
%\begin{lemma}  [Fourier Slice theorem for $R$] \label{scalar Fourier slice}
%For any $f\in L^1(\mathbb R^n)$,
%$$
%\hat f (r\omega)=\int_{\mathbb R} e^{-isr} R_0(f)(s,\omega)\, \d s, \quad \forall r\in \mathbb R,\, w\in \mathbb S^{n-1}.
%$$
%\end{lemma}

\begin{lemma}[Fourier Slice theorem for $R$] For $f\in L^1(\R^n)\cup L^2(\R^n)$, 
\begin{equation}\label{tensor Fourier slice 2}
\int_{\mathbb R} e^{-isr} Rf(s,\omega)\, \d s = \hat f_{ij}(r\omega)(\delta^{ij}-\omega^i \omega^j), \quad \forall r\in \mathbb R,\, \forall \omega\in \mathbb S^{n-1}. 
\end{equation}
\end{lemma}

\begin{proof} 
Write
\begin{align}
\int_{\mathbb R} e^{-irs} Rf(s,\omega)\, \d s &=  \int_{\R} \int_{x\cdot \omega=s}e^{-irs}  (\tr f - f_{ij} \omega^i\omega^j) \, \d \mu\, \d s\\
&=  \int_{\R} \int_{x\cdot \omega=s}e^{-i r\omega\cdot x}  (\tr f - f_{ij} \omega^i\omega^j) \, \d \mu\, \d s\\
&=  \int_{\R^n} e^{-i r\omega\cdot x}  (\tr f - f_{ij} \omega^i\omega^j) \, \d x.
\end{align}
\end{proof}
Recall that the symmetric differentiation of a one-form $v$ is given by
\[
(d^s v)_{ij}=\frac{1}{2}(\p_i v_j+\p_j v_i),
\]
so it's easy to check that
\[
(\widehat{d^s v})_{ij} (\xi)\Big(\delta^{ij}-\frac{\xi^i \xi^j}{|\xi|^2}\Big)=0.
\]
Therefore we can show directly by \eqref{tensor Fourier slice 2} that the potential tensors are in the kernel of $R$ (through the extension of a compactly supported potential tensor by zero to the whole space $\mathbb R^n$). Of course, this is already verified in the appendix even for the Riemannian case.

On the other hand, \eqref{tensor Fourier slice 2} also shows that there are additional elements in the kernel of $R$. Let $\div$ be the adjoint to $-\d^s$ under the $L^2$ inner product.  Symmetric tensors $f$ satisfying the equation $\div f=0$ are called divergence free (or solenoidal) tensors.  We denote by $\div$ the divergence of one-forms, as well. 
We have the following orthogonal decomposition of $L^2$ symmetric tensor fields on $\mathbb R^n$, see \cite{Sh-book}.

\begin{lemma}\label{decomposition on R^n 1}
For any symmetric 2-tensor $f\in L^2(\mathbb R^n)$ of compact support, there exists uniquely determined symmetric 2-tensor $f^{s}_{\mathbb R^n}\in L^2(\mathbb R^n)$ and 1-form $v_{\mathbb R^n}\in H^1_{\rm loc}(\mathbb R^n)$, smooth outside $\supp f$, such that 
\[
f=f^{s}_{\mathbb R^n}+\d^s v_{\mathbb R^n},\quad \div f^{s}_{\mathbb R^n}=0,\quad v_{\mathbb R^n}\to 0 \,\, \mbox{as} \,\, |x|\to \infty.
\]
Moreover, the following estimates hold for $|x|$ large enough
\[
|f^{s}_{\mathbb R^n}(x)|+|\d^s v_{\mathbb R^n}(x)|\leq C(1+|x|)^{-n},\quad |v_{\mathbb R^n}(x)|\leq C(1+|x|)^{1-n}.
\]
\end{lemma}

We already know that $R(\d^s v_{\mathbb R^n})=0$. We consider the $L^2$ solenoidal symmetric 2-tensors, i.e., those $f$ satisfying 
\[
(\div f)_i=\p^j f_{ij}=0,\, i=1,\dots, n.
\]
These equations imply  
\[
\xi^j \hat f_{ij}(\xi)=0,\, i=1,\dots,n.
\]
 Therefore
\begin{equation}   \label{FT}
\hat f_{ij}(\xi) \Big(\delta^{ij}-\frac{\xi^i \xi^j}{|\xi|^2}\Big)=\hat f_{ij}(\xi)\delta^{ij}={\tr(\hat f)}(\xi),
\end{equation}
and by \eqref{tensor Fourier slice 2} again, this implies that trace-free solenoidal 2-tensors are in the kernel of $R$ (when defined on the whole space $\mathbb R^n$). They are called transverse--tracefree (TT) tensors in relativity. 

\begin{proposition}\label{decomposition on R^n 2}
For any symmetric 2-tensor $f\in L^2(\mathbb R^n)$ of compact support, there exists uniquely determined symmetric 2-tensor $f^\mathsf{TT}_{\mathbb R^n}\in L^2(\mathbb R^n)$, 1-form $v_{\mathbb R^n}\in H^1_\textrm{\rm loc}(\mathbb R^n)$ and a function $\lambda_{\mathbb R^n} \in L^2(\mathbb R^n)$, smooth outside $\supp f$, such that 
\[
f=f^\mathsf{TT}_{\mathbb R^n}+\d^s v_{\mathbb R^n}+\lambda_{\mathbb R^n}{\bf e},\quad \div f^\mathsf{TT}_{\mathbb R^n}=0, \quad \tr (f^\mathsf{TT}_{\mathbb R^n})=0, 
\]
where ${\bf e}$ is the Euclidean metric, and for $|x|$ large enough, 
\begin{equation}\label{est}
|f^\mathsf{TT}_{\mathbb R^n}(x)|+|\d^s v_{\mathbb R^n}(x)|+|\lambda_{\mathbb R^n}(x)|\leq C(1+|x|)^{-n},\quad |v_{\mathbb R^n}(x)|\leq C(1+|x|)^{1-n}.
\end{equation} 
\end{proposition}
\begin{proof}
The proof is provided in \cite{SU-book}, and it follows the proof in a bounded domain with boundary given in \cite{Sh-2D}, see Proposition~\ref{pr_M}.
\end{proof}

Observe that \r{est} implies that $R$ is well-defined on each component. 
\begin{proposition}\label{pr_R^n}
	With $f$ as in Proposition~\ref{decomposition on R^n 2}, assume $Rf=0$. Then $\lambda_{\mathbb R^n}=0$, i.e.,
	\[
	f=f^\mathsf{TT}_{\mathbb R^n}+\d^s v_{\mathbb R^n},\quad \text{$\div f^\mathsf{TT}_{\mathbb R^n}=0$, \,$\tr  (f^\mathsf{TT}_{\mathbb R^n})=0$}, %\quad v_{\mathbb R^n}\to 0 \,\, \mbox{as} \,\, |x|\to \infty,
	\] 
with 	$f^\mathsf{TT}_{\mathbb R^n}$, $v_{\mathbb R^n}$ having the regularity in Proposition~\ref{decomposition on R^n 2}, satisfying  estimate \eqref{est}. On the other hand, for every such $f$, $f^\mathsf{TT}_{\mathbb R^n}$, $v_{\mathbb R^n}$, we have $Rf=0$. 
\end{proposition}

\begin{proof}
The second statement follows from \eqref{FT}. For the first one, we get first $0=Rf = R (\lambda_{\mathbb R^n}{\bf e})$, which is just the Euclidean Radon transform of the function $(n-1)\lambda_{\mathbb R^n}$, therefore, $\lambda_{\mathbb R^n}=0$. 
\end{proof}

 The representation \eqref{tensor Fourier slice 2} indicates that $\Ker R$ might consist of the kernel of the differential operator $P$ introduced in the Introduction in \r{P}, 
with the right boundary conditions or conditions at infinity.  

We denote the standard Radon transform of scalar functions by $R_0$.

\begin{lemma} \label{lemma1}  
For every symmetric two-tensor field $f\in \mathcal{E}'(\R^n)$, 
\[
R_0Pf = -\p_s^2 Rf.
\]
In particular, such an $f$ is in $\Ker R$ if and only it is in $\Ker P$. 
\end{lemma}

\begin{proof}
We can assume $\omega=e_n$ without loss of generality.	Let $f\in C_0^\infty(\R^n)$ first. We have
\begin{align}
R_0 Pf(s,e_n)& = \int_{x^n=s}(-\Delta \tr f+ \p^i \p^j f_{ij})\, \d x'
=  \int_{x^n=s}(-\p_n^2 \tr f+ \p_n  \p^j f_{nj})\, \d x'\\
&=\int_{x^n=s}(-\p_n^2 \tr f+ \p_n^2 f_{nn})\, \d x' = -\p_n^2 Rf(s,e_n). 
\end{align}
We used  Green's formula once (integral of $\Delta' (\tr f)=\p_1^2 \tr f+\cdots+\p_{n-1}^2 \tr f$ on $x^n=s$ is zero), and the divergence theorem twice. 

To prove the first statement for $f\in \mathcal{E}'(\R^n)$, given such an $f$, we approximate it with $f_n\in C_0^\infty(\R^n)$ in $\mathcal{D}'(\R^n)$.

Now, let $f\in \Ker P\cap  \mathcal{E}'(\R^n)$. Then $\p_s^2Rf=0$, and since $Rf$ is compactly supported, we get $Rf=0$. Assume  $f\in \Ker R\cap  \mathcal{E}'(\R^n)$. Then $R_0Pf=0$ with $Pf\in  \mathcal{E}'(\R^n)$, which implies $Pf=0$.
\end{proof}

\section{\texorpdfstring{$R$}{R} restricted to a bounded domain} 
It is known that on a compact Riemannian manifold $(M,g)$ with boundary, there is a unique decomposition of $L^2$ symmetric 2-tensor $f$ \cite[Theorem 3.3]{Sh-2D}:
\begin{proposition}\label{pr_M} For every symmetric 2-tensor field $f\in H^k(M)$, $k=0,1,\dots$, 
\begin{equation} \label{dec_Omega} 
f=\d^s v+ \lambda g+ f^{\mathsf{TT}},\quad v|_{\p M}=0,\quad \div f^\mathsf{TT}=0,\quad \tr(f^\mathsf{TT})=0,
\end{equation}
for uniquely determined  one-form $v\in H^{k+1}$, a scalar function $\lambda\in H^k$, and a symmetric 2-tensor $f^\mathsf{TT}\in H^k$, depending continuously on $f\in H^k$. 
\end{proposition}
In local coordinates
\[
(\d^s v)_{ij}=\frac{1}{2}(\nabla_i v_j+\nabla_j v_i),
\]
where 
\[
\nabla_i v_j=\frac{\p v_j}{\p x^i}-\Gamma_{ij}^k v_k,
\]
$\Gamma_{ij}^k$ are the Christoffel symbols. 
$$(\div f)_i=g^{jk}\nabla_k f_{ij}, \quad i=1,\cdots, n,$$
and
$$\tr f=g^{ij}f_{ij},$$
where $g^{-1}=(g^{ij})$ represents the dual metric of $g$.

In particular $f^\mathsf{TT}$ is divergence-free and trace-free. In our case, $g$ is Euclidean, and $M$ is the closure of a bounded smooth domain $\Omega\subset \mathbb R^n$.

First, if $f=\d^s v$ is potential, $\d^s$ commutes with the extension as zero because $v=0$ on $\p\Omega$, and such fields are in $\Ker P$ (as well as in $\Ker R$). 

Let $\nu$ be the unit outer normal vector along the boundary $\p\Omega$. Given a symmetric $2$-tensor $f$, we denote 
$f_{\nu\nu}=f(\nu,\nu)$ and $f_{\nu T}$ is the tangential part (w.r.t. the boundary $\p\Omega$) of the $1$-form $f(\nu,\,\cdot\,)$. Next, $\div_{\p\Omega}$ is the tangential divergence (i.e. the divergence operator of the boundary $\p\Omega$).

\begin{theorem} \label{thm_M1}
The symmetric two-tensor field $f\in H^2(\Omega)$ 
belongs to $\Ker R$ if and only if its decomposition by Proposition~\ref{pr_M} satisfies the following: 
\[
\Delta\lambda=0 \quad \text{in $\Omega$},
\]
\[ f^\mathsf{TT}_{\nu\nu} = (n-1)\lambda, \quad \div_{\bo}  f_{\nu T}^\mathsf{TT} =(n-1)\partial_\nu\lambda\quad \text{on $\bo$.}
\]
\end{theorem}

We need the following lemma for the proof of the theorem. In it,  we use the standard convention
\begin{equation}\label{DiracDefs}
\langle \delta_{\partial\Omega},\psi\rangle
=
\int_{\partial\Omega}\psi\,dS,
\qquad
\langle \partial_\nu\delta_{\partial\Omega},\psi\rangle
=
-\int_{\partial\Omega}\partial_\nu\psi\,dS
\end{equation}
for any test function \(\psi\in C^\infty_0(\mathbb R^n)\). Next, $(\div f)_\nu=(\div f)(\nu)$, and the notation $\left.h\right|_{\partial \Omega}\partial _\nu\delta_{\partial \Omega}$ denotes the distribution in $\mathcal{D}^\prime \left(\mathbb{R}^n \right)$ given by
$$\langle \left.h\right|_{\partial \Omega}\partial _\nu\delta_{\partial \Omega},\psi\rangle =-\int_{\partial \Omega}{\left.h\right|_{\partial \Omega}\partial _\nu\psi\ dS}=-\int_{\partial \Omega}{h\ \partial _\nu\psi\ dS}.$$
for any $\psi\in C_0^\infty \left(\mathbb{R}^n \right)$.

\begin{lemma} \label{lemma_Pf}
For $f\in H^2(\Omega)$, 
\begin{equation}\label{P_Of_fChi}
\begin{aligned}
P(f\chi )
={}&
(Pf)\chi  \\
&+
\left[
\partial_\nu\operatorname{tr}f
-
(\operatorname{div}f)_\nu
-
\operatorname{div}_{\partial\Omega}(f_{\nu T})
\right]\delta_{\partial\Omega} \\
&+
\left. \left(\operatorname{tr}{f}-f_{\nu\nu} \right)\right|_{\partial \Omega}
\partial_\nu\delta_{\partial\Omega}.
\end{aligned}
\end{equation}
\end{lemma}

\begin{proof} 
Let $f\in C^2(\bar\Omega)$ first. Set
\[
\tau=\operatorname{tr} f,\qquad
Pf=-\Delta \tau+\p^i\p^j f_{ij}.
\]
For a test function \(\psi\),
\[
\langle P(f\chi),\psi\rangle
=
-\int_\Omega \tau\,\Delta\psi\,dx
+
\int_\Omega f_{ij}\p^i\p^j\psi\,dx .
\]
Integrating the first term by parts twice gives
\[
-\int_\Omega \tau\,\Delta\psi\,dx
=
-\int_\Omega (\Delta\tau)\psi\,dx
+
\int_{\partial\Omega}\partial_\nu\tau\,\psi\,dS
-
\int_{\partial\Omega}\tau\,\partial_\nu\psi\,dS .
\]
For the second term,
\[
\begin{aligned}
\int_\Omega f_{ij}\p^i\p^j\psi\,dx
&=
-\int_\Omega (\p^i f_{ij})\p^j\psi\,dx
+
\int_{\partial\Omega} f_{\nu j}\p^j\psi\,dS  \\
&=
\int_\Omega \p^j\p^i f_{ij}\,\psi\,dx
-
\int_{\partial\Omega}(\operatorname{div}f)_\nu\psi\,dS
+
\int_{\partial\Omega} f_{\nu j}\p^j\psi\,dS .
\end{aligned}
\]
On \(\partial\Omega\),
\[
f_{\nu j}\p^j\psi
=
f_{\nu\nu}\partial_\nu\psi
+
\langle f_{\nu T},\nabla_{\partial\Omega}\psi\rangle .
\]
Since \(\partial\Omega\) is closed,
\[
\int_{\partial\Omega}
\langle f_{\nu T},\nabla_{\partial\Omega}\psi\rangle\,dS
=
-\int_{\partial\Omega}
\operatorname{div}_{\partial\Omega}(f_{\nu T})\psi\,dS .
\]
Therefore
\[
\begin{aligned}
\langle P(f\chi),\psi\rangle
={}&
\int_\Omega
\left(-\Delta\tau+\p^j\p^i f_{ij}\right)\psi\,dx \\
&+
\int_{\partial\Omega}
\left[
\partial_\nu\tau
-
(\operatorname{div}f)_\nu
-
\operatorname{div}_{\partial\Omega}(f_{\nu T})
\right]\psi\,dS \\
&+
\int_{\partial\Omega}
(f_{\nu\nu}-\tau)|_{\bo}\,\partial_\nu\psi\,dS .
\end{aligned}
\]
For a general $f\in H^2(\Omega)$, approximate $f$ by a sequence $\{f_n\in C^2(\bar{\Omega})\}$ in $H^2$, apply \r{P_Of_fChi} to each $f_n$, and then take the limit of both sides in $\mathcal{D}^\prime(\mathbb{R}^n)$.
\end{proof}
 
\begin{proof}[Proof of Theorem~\ref{thm_M1}] 
We have $R(\chi  d^sv)=0$, where $\chi$ actually stands for extension from $\Omega$ to $\R^n$ as zero, and we used the fact that $\d^s$ commutes with it on $H_0^1(\Omega)$  (or see Corollary \ref{PotInKerR} below). Similarly, $P(\chi \d^s v)=0$. 
We need to analyze $R$ and $P$ applied to  $\chi f^\mathsf{TT}$ and $\chi \lambda\mathbf{e}$. For the latter, by Lemma~\ref{lemma_Pf}, we have
\[
P(\lambda\mathbf{e}\chi) = (n-1)[ -(\Delta\lambda)\chi + (\partial_\nu\lambda) \delta_{\bo}+ \lambda |_{\p\Omega} \partial_\nu\delta_{\bo}].
\]
For the TT part, we have
\[
P(f^\mathsf{TT}\chi) = -\operatorname{div}_{\partial\Omega}(f_{\nu T}^\mathsf{TT})
\delta_{\partial\Omega} - f_{\nu\nu}^\mathsf{TT}|_{\p\Omega} \p_\nu \delta_{\partial\Omega}.
\]
Combining those two with Lemma~\ref{lemma1}, we complete the proof. 
\end{proof}

\begin{remark}
It is not a priori clear how rich the space of $(\lambda,f^\mathsf{TT})$ satisfying the conclusions of Theorem~\ref{thm_M1} is. When $\lambda=0$, such non-trivial TT tensors exist, as mentioned above. We do not know if $\lambda$ can be non-zero. 
\end{remark}

If $f$ is TT a priori, it is in $\Ker R$ if and only if 
\begin{equation}   \label{cond2}
  f_{\nu\nu}=0, \quad \div_{\bo}  f_{\nu T} =0. 
 \end{equation} 
 In dimension \(3\), non-zero compactly supported TT tensor fields exist, see \cite[Appendix~B]{Corvino}, \cite[Proposition~3.10]{Gicquaud} for explicit constructions. They  can be constructed by
applying the linearized Cotton--York operator to compactly supported
trace-free symmetric \(2\)-tensors;  see Beig~\cite{Beig_TT}. 
Moreover, on compact manifolds without boundary, in suitable weighted Sobolev topologies, compactly supported TT tensors are dense in the corresponding TT space; see Remark 9.9 in Delay \cite{Delay} and Dahl--Kröncke \cite{DahlKr}

The theorem shows that the whole space equivalent to Proposition~\ref{pr_R^n} does not hold since we need extra conditions on $f^\mathsf{TT}$ to guarantee that it is in $\Ker R$. If those conditions are not met,   $Rf=0$ need not force the conformal component in the bounded-domain decomposition to vanish.

We can use the standard potential-solenoidal decomposition instead. 

\begin{theorem} \label{thm_M2} The symmetric two-tensor field $f\in H^2(\Omega)$ belongs to $\Ker R$ if and only if its standard potential-solenoidal decomposition $f=\d^s v+f^s$ in $\Omega$   satisfies the following: 
\[
\Delta\tr f^s=0 \quad \text{in $\Omega$},
\]
\[   f^s_{\nu\nu} = \tr f^s, \quad \div_{\bo} f^s_{\nu T} = \partial_\nu\tr f^s \quad \text{on $\bo$.}
\]
\end{theorem}

The proof follows directly from Lemma~\ref{lemma1} and Lemma~\ref{lemma_Pf}. The theorem does not require $\tr f^s=0$ (and does not separate it from a conformal term) but it requires the trace to be harmonic. It is not clear whether solutions for $f^s$ with non-zero traces exist. If $\tr f^s=0$, then such tensors reduce to the ones in Theorem~\ref{thm_M1}.

\appendix 
\section{Linearization of the area functional}\label{appendix}
We recall the linearization of the minimal surface functional, derived in  \cite{busch2025_minimal}, as well, and show that potential fields belong to its kernel. 
Let $(M,g)$ be a compact Riemannian manifold with boundary of dimension $n \geq 2$. For $1\le k<n$, given a $(k-1)$-dimensional closed submanifold $\gamma$ of $\p M$, under some conditions, see, e.g.,  \cite{busch2025_minimal} for $k=n-1$, we can find a $k$-dimensional minimal surface $\Sigma$ (w.r.t. the metric $g$) in $M$ such that $\p\Sigma=\gamma$. 
We denote the area of the minimal surface $\Sigma$ by $A(\Sigma)$. We consider the linearization of the area $A(\Sigma)=A_g(\Sigma)$ w.r.t. the metric $g$.

When $k=1$, $\Sigma$ is a geodesic on Riemannian manifold $M$, and $A(\Sigma)$ is its length. The problem then is reduced to the usual boundary rigidity problem of Riemannian manifolds, which has been extensively studied, see the recent survey \cite{SUVZ-survey} and the references therein for more details.

\begin{proposition}
Let $f$ be an arbitrary symmetric 2-tensor, and define $g^s=g+sf$ for $s\in (-\epsilon, \epsilon)$ with $0< \epsilon\ll 1$. Let $\gamma$ be a $(k-1)$-dimensional closed submanifold in $\p M$, for each $s$, we denote the unique $k$-dimensional minimal surface w.r.t. $g^s$, whose boundary is $\gamma$, by $\Sigma_s$. Let $A_{g^s}(\Sigma_s)$ be the area of $\Sigma_s$ w.r.t. $g^s$. Define the embedding $\psi: \Sigma_0\to M$. 
Then 
\begin{equation}\label{tensor Radon}
\frac{\d}{\d s}\Big|_{s=0} A_{g^s}(\Sigma_s)=\frac{1}{2}\int_{\Sigma_0} \tr_{\psi^*(g)} \Big(\psi^*(f)\Big)\, \d\mu,
\end{equation}
where $\d\mu$ is the volume form of $\Sigma_0$.
\end{proposition}
Notice that $\psi^*(g)=\left(\psi^*(g)_{\alpha\beta}\right)$ is the induced metric on $\Sigma_0$, whose dual metric is $\psi^*(g)^{-1}=\left\{\psi^*(g)^{\alpha\beta}\right\}$. In local coordinates $\tr_{\psi^*(g)}  \psi^*(f) := \psi^*(g)^{\alpha\beta} \psi^*(f)_{\alpha\beta}$ is the trace of the pull-backed tensor $\psi^*(f)$ on the minimal surface $\Sigma_0$. Given an orthonormal basis $\{e_1,\cdots, e_k\}$ for $\Sigma_0$, the trace is simply $\sum_{\alpha=1}^k f(e_\alpha, e_\alpha)$.
\begin{proof}
For $s\in (-\epsilon,\epsilon)$, we consider proper embeddings $\psi_s: \Sigma_0\to M$ with $\psi_s(\Sigma_0)=\Sigma_s$ and $\psi_s|_{\gamma}=\id$,  so $\psi_0=\psi$. 
%We simply denote $f_0$ by $f$.  
Let $\{x^1,\cdots, x^n\}$ be local coordinates on $M$ and $\{y^1,\cdots, y^k\}$ be local coordinates for $\Sigma_0$. So $\psi_s(y^1,\cdots, y^k)=(\psi_s^1,\cdots, \psi_s^n)$ and 
$$A_{g^s}(\Sigma_s)=\int_{\Sigma_0} \sqrt{ \det \psi_s^* (g^s)}\, \d y=\int_{\Sigma_0} \sqrt{ \det \Big(g^s_{ij}(\psi_s(y))\frac{\p \psi_s^i}{\p y^\alpha}\frac{\p \psi_s^j}{\p y^\beta}\Big)}\, \d y.$$
Then
$$\frac{\d}{\d s}\Big|_{s=0} A_{g^s}(\Sigma_s)=\int_{\Sigma_0} \frac{1}{2 \sqrt{\det (\psi^*(g))}} \frac{\d}{\d s}\Big|_{s=0} \det \Big(g^s_{ij}(\psi_s(y))\frac{\p \psi_s^i}{\p y^\alpha}\frac{\p \psi_s^j}{\p y^\beta}\Big)\, \d y.$$
Notice that $g^s=g+sf$, it is not difficult to check that 
\begin{equation*}
\begin{split}
\frac{\d}{\d s}\Big|_{s=0} \det \Big(g^s_{ij}(\psi_s(y))\frac{\p \psi_s^i}{\p y^\alpha}\frac{\p \psi_s^j}{\p y^\beta}\Big)=\frac{\d}{\d s}& \Big|_{s=0}  \det  \Big(g_{ij}(\psi_s(y))\frac{\p \psi_s^i}{\p y^\alpha}\frac{\p \psi_s^j}{\p y^\beta}\Big)\\
&+\frac{\d}{\d s}\Big|_{s=0} \det \Big(g^s_{ij}(\psi(y))\frac{\p \psi^i}{\p y^\alpha}\frac{\p \psi^j}{\p y^\beta}\Big).
\end{split}
\end{equation*} 
By the first variational formula for the area and the fact that $\Sigma_0$ is a minimal surface w.r.t. $g$, we get that 
\begin{align*}
\int_{\Sigma_0} \frac{1}{2 \sqrt{\det (\psi^*(g))}}\frac{\d}{\d s}\Big|_{s=0} &\det\Big(g_{ij}(\psi_s(y))\frac{\p \psi_s^i}{\p y^\alpha}\frac{\p \psi_s^j}{\p y^\beta}\Big)\, \d y\\
&=\int_{\Sigma_0} \frac{\d}{\d s}\Big|_{s=0}\sqrt{ \det\Big(g_{ij}(\psi_s(y))\frac{\p \psi_s^i}{\p y^\alpha}\frac{\p \psi_s^j}{\p y^\beta}\Big)}\,\d y=0.
\end{align*}
Now by Jacobi's formula
%\begin{equation*}
\begin{align}
\frac{\d}{\d s}\Big|_{s=0} A_{g^s}(\Sigma_s)&=\int_{\Sigma_0} \frac{1}{2 \sqrt{\det (\psi^*(g))}} \frac{\d}{\d s}\Big|_{s=0} \det \Big(g^s_{ij}(\psi(y))\frac{\p \psi^i}{\p y^\alpha}\frac{\p \psi^j}{\p y^\beta}\Big)\, \d y\\
&=\int_{\Sigma_0} \frac{1}{2}\sqrt{\det (\psi^*(g))}\, \tr \Big\{(\psi^*(g))^{-1} \frac{\d}{\d s}\Big|_{s=0} \Big(g^s_{ij}(\psi(y))\frac{\p \psi^i}{\p y^\alpha}\frac{\p \psi^j}{\p y^\beta}\Big)\Big\}\, \d y\\
&=\frac{1}{2}\int_{\Sigma_0} \tr \Big\{(\psi^*(g))^{-1} \Big(f_{ij}(\psi(y))\frac{\p \psi^i}{\p y^\alpha}\frac{\p \psi^j}{\p y^\beta}\Big)\Big\}\, \d\mu\\
&= \frac{1}{2}\int_{\Sigma_0} \tr_{\psi^*(g)} \Big(\psi^*(f)\Big)\, \d\mu.
\end{align}
\end{proof}

We consider the r.h.s.\ of \eqref{tensor Radon} as a generalized Radon transform of symmetric 2-tensors. Given a $k$-dimensional minimal surface $\Sigma$ of $(M,g)$, let $\iota:\Sigma\to M$ be the classical inclusion. For the sake of simplicity, we simplify the notation $\tr_{\iota^*(g)}\Big(\iota^*(f)\Big)$ as $\tr_{\Sigma}(f)$. We define

\begin{definition}\label{def1} 
\[
Rf(\Sigma):=\int_{\Sigma} \tr_{\iota^*(g)}\Big(\iota^*(f)\Big)\, \d\mu=\int_{\Sigma} \tr_{\Sigma}(f)\, \d\mu,
\]
with $\d\mu$ the volume form on $\Sigma$. 
\end{definition}
It's easy to see that in the conformal case, i.e. $g^s=(1+sc^2)g$, the above linearization gives exactly the usual Radon type transform of a scalar function over minimal surfaces of $M$.

Recall the definition of the potential symmetric 2-tensor fields $\d^sv$, where $\d^s=\sigma \nabla$ is the symmetric differentiation with the Levi-Civita connection $\nabla$, for some 1-form $v$ with $v|_{\p M}=0$.

%=============
\begin{proposition}\label{prop_potential_minimal}
Let \(\Sigma \subset (M,g)\) be an embedded \(k\)-dimensional submanifold,
with inclusion \(\iota:\Sigma\to M\). Let \(v\) be a one-form on \(M\). 
%and let \(h=d^s v\). 
Then
\[
\tr_\Sigma(d^s v)
=
\div_\Sigma(\iota^*v)-v(\vec H),
\]
where \(\operatorname{div}_\Sigma\) denotes the divergence on one-forms
with respect to the induced metric, i.e.
\[
\operatorname{div}_\Sigma \alpha
=
(\iota^*g)^{ab}\nabla^\Sigma_a \alpha_b, 
\] 
and \(\vec H\) is the mean curvature
vector of \(\Sigma\) in \(M\), defined by
\[
\vec H=\sum_{\alpha=1}^k \big(\nabla_{e_\alpha}e_\alpha\big)^\perp
\]
for any local \(\iota^*g\)-orthonormal frame
\(e_1,\dots,e_k\) of \(T\Sigma\). Here $\perp$ denotes the normal part of a vector w.r.t. the submanifold $\Sigma$. In particular, if \(\Sigma\) is minimal,
then
\[
\tr_{\Sigma}(d^s v)
=
\div_{\Sigma}(\iota^*v).
\]
\end{proposition}

\begin{proof}
Let \(e_1,\dots,e_k\) be a local \(\iota^*g\)-orthonormal frame of
\(T\Sigma\). By definition,
\[
(d^s v)(X,Y)
=
\frac12\big((\nabla_Xv)(Y)+(\nabla_Yv)(X)\big).
\]
Therefore
\[
\tr_{\Sigma}(d^s v)
=
\sum_{\alpha=1}^k (d^s v)(e_\alpha,e_\alpha)
=
\sum_{\alpha=1}^k (\nabla_{e_\alpha}v)(e_\alpha).
\]
Let \(V=v^\sharp\) be the vector field on \(M\) dual to \(v\). Then
\[
(\nabla_{e_\alpha}v)(e_\alpha)
=
g(\nabla_{e_\alpha}V,e_\alpha).
\]
Decompose \(V\) along \(\Sigma\) into tangential and normal parts:
\[
V=V^T+V^\perp.
\]
Since \(V^T=(\iota^*v)^\sharp\), we have
\[
\sum_{\alpha=1}^k g(\nabla_{e_\alpha}V^T,e_\alpha)
=
\div_{\Sigma}\big((\iota^*v)^\sharp\big).
\]
For the normal part, using \(g(V^\perp,e_\alpha)=0\), we get
\[
0
=
e_\alpha g(V^\perp,e_\alpha)
=
g(\nabla_{e_\alpha}V^\perp,e_\alpha)
+
g(V^\perp,\nabla_{e_\alpha}e_\alpha).
\]
Hence
\[
g(\nabla_{e_\alpha}V^\perp,e_\alpha)
=
-
g\big(V^\perp,(\nabla_{e_\alpha}e_\alpha)^\perp\big).
\]
Summing over \(\alpha\), we obtain
\[
\sum_{\alpha=1}^k g(\nabla_{e_\alpha}V^\perp,e_\alpha)
=
-
g(V^\perp,\vec H).
\]
Since \(\vec H\) is normal, \(g(V^\perp,\vec H)=g(V,\vec H)=v(\vec H)\).
Therefore
\[
\tr_{\Sigma}(d^s v)
=
\div_{\Sigma}\big((\iota^*v)^\sharp\big)
-
v(\vec H).
\]
If \(\Sigma\) is minimal, then \(\vec H=0\), and the claimed identity follows.
\end{proof}
%===========

\begin{corollary}\label{PotInKerR}
Potential fields belong to $\Ker R$. 
\end{corollary}

\begin{proof}
 Since $\Sigma$ is a minimal surface, we have $\vec{H}=0$. By the divergence theorem, 
\begin{align}\label{kernel of R}
R(\d^sv)(\Sigma)=\int_{\Sigma} \tr_{\Sigma}(\d^s v)\, \d\mu=\int_{\Sigma} \div_{\Sigma} (\iota^*(v))\, \d\mu=\int_{\gamma} v(\nu)\, d\sigma=0,
\end{align}
where $\nu$ is the unit outer normal vector along $\p \Sigma=\gamma\subset \p M$ of the minimal surface $\Sigma$ and $d\sigma$ is the volume form on $\gamma$. 
\end{proof}

\bibliographystyle{abbrv}
% %
%\bibliography{../myreferences.bib}
%\bibliography{myreferences.bib}

\end{document}